\documentclass[a4]{amsart}

\input xypic 
\input xy 
\xyoption{all} 
\usepackage{amssymb}
\usepackage{bbm} 
\usepackage{xcolor} 
\usepackage{tikz}
\usepackage{tikz-cd}
\usepackage{verbatim}
\newtheorem*{theorem*}{Theorem}
\newtheorem{theorem}{Theorem}[section]
\newtheorem{proposition}[theorem]{Proposition}
\newtheorem{lemma}[theorem]{Lemma}
\newtheorem{corollary}[theorem]{Corollary}

\theoremstyle{definition}

\newtheorem{remark}[theorem]{Remark}

\newcommand{\conn}{\ensuremath{\#}}

\newcommand{\rhlgy}[1]{\ensuremath{\widetilde{H}_{*}(#1)}}

\newcounter{bean}

\newcommand{\namedright}[3]{\ensuremath{#1\stackrel{#2}
 {\longrightarrow}#3}}
\newcommand{\nameddright}[5]{\ensuremath{#1\stackrel{#2}
 {\longrightarrow}#3\stackrel{#4}{\longrightarrow}#5}}
\newcommand{\namedddright}[7]{\ensuremath{#1\stackrel{#2}
 {\longrightarrow}#3\stackrel{#4}{\longrightarrow}#5
  \stackrel{#6}{\longrightarrow}#7}}

\begin{document}


\title{Local Inertness of Poincar\'{e} duality complexes} 

\author{Samik Basu}
\address{Stat Math Unit, Indian Statistical Institute Kolkata 700108, India}
\email{samik.basu2@gmail.com; samikbasu@isical.ac.in}

\author{Sebastian Chenery}
\address{University of Bristol, School of Mathematics, Fry Building, Woodland Road, Bristol, BS8 1UG}
\email{seb.chenery@bristol.ac.uk}

\author{Ruizhi Huang}
\address{State Key Laboratory of Mathematical Sciences and Institute of Mathematics, Academy of Mathematics
and Systems Science, Chinese Academy of Sciences, Beijing 100190, China
   \newline
   \indent School of Mathematical Sciences, University of Chinese Academy of Sciences, Beijing 100049, China}
\email{huangrz@amss.ac.cn}

\author{Lewis Stanton}
\address{Mathematical Sciences, University of Southampton, Southampton 
   SO17 1BJ, United Kingdom}
\email{L.R.Stanton@soton.ac.uk}

\author{Stephen Theriault}
\address{Mathematical Sciences, University of Southampton, Southampton 
   SO17 1BJ, United Kingdom}
\email{S.D.Theriault@soton.ac.uk}



\begin{abstract} 
We prove that, under certain homological conditions, the attaching map of the top cell of a Poincar\'{e} duality complex is inert when localised away from a finite set of primes. This improves on a result of F\'elix and Tanr\'e in these cases. As an additional application of the methods, we give a loop space decomposition of simply-connected $6$-dimensional Poincar\'{e} duality complexes satisfying certain hypotheses. We also show that, under the hypotheses of the inertness theorem, the $(n-1)$-skeleton of an $n$-dimensional Poincar\'{e} duality complex satisfies the hyperbolic form of Moore's Conjecture after localising away from an explicit finite set of primes, and use this to obtain new examples of \(p\)-local maps between spheres that are not inert. 
\end{abstract} 

\medskip 

\maketitle

\section{Introduction}
\label{sec:intro}

A major problem in homotopy theory is to understand the effect on homotopy groups of attaching a cell. In the context of Poincar\'{e} duality complexes, this has attracted much recent attention. Let $M$ be a simply-connected Poincar\'{e} duality complex of dimension $n$. Denoting by $\overline{M}$ the $(n-1)$-skeleton of $M$, there is a homotopy cofibration $S^{n-1} \xrightarrow{f} \overline{M} \xrightarrow{i} M$, where $f$ is the attaching map for the top cell of $M$ and $i$ is the skeletal inclusion. For $\mathbb{K}$ a field, the map $f$ is called $\mathbb{K}$-inert if $(\Omega f)_*:H_*(\Omega \overline{M};\mathbb{K}) \to H_*(\Omega M;\mathbb{K})$ is surjective. 

Halperin and Thomas \cite{HT} proved the remarkable fact that $f$ is $\mathbb{Q}$-inert if the cohomology ring $H^*(M;\mathbb{Q})$ is not generated by a single element. Later, F\'elix and Tanr\'e \cite{FT} proved an analogous result for fields of finite characteristic in the case that $M$ is a closed manifold. In particular, they proved the following, where for a finite $CW$-complex $X$ and field $\mathbb{K}$, $B_{X,\mathbb{K}} = \sum_{i \geq 1} \text{rank}(H_i(X;\mathbb{K}))$.  

\begin{theorem*}
    Let \(M\) be a closed, \(k\)-connected, \(n\)-dimensional manifold. If the cohomology ring \(H^*(M;\mathbb{K})\) is not generated by a single element, then there exists an integer \(N\) such that for every field \(\mathbb{K}\) of characteristic \(p > N\), the attaching map for the top cell of \(M\) is \(\mathbb{K}\)-inert. In general, one may take
    \[
        N = n + k \cdot B_{M,\mathbb{K}}.
    \]
    If \(M\) is instead \(2k\)-connected but not \((2k+1)\)-connected, one may take \(N = n\).
\end{theorem*}

In \cite{T1}, an integral version of inertness was introduced. In particular, the attaching map for the top cell of $M$ is inert if $\Omega i$ has a right homotopy inverse. Rationally, inertness in this sense is equivalent to $\mathbb{Q}$-inertness. However, inertness is stronger than the analogous `$\mathbb{Z}$-inertness'. For example, $\overline{SU(3)}\simeq\Sigma\mathbb{C}P^{2}$, and if $i\colon\Sigma\mathbb{C}P^{2}\longrightarrow SU(3)$ is the inclusion then $(\Omega i)_{\ast}$ is the abelianisation of the tensor algebra $T(\rhlgy{\mathbb{C}P^{2}})$ so it is surjective in integral homology. However, $\Omega i$ does not have a right homotopy inverse, for this would imply that $\Sigma\Omega i$ has a right homotopy inverse, which when compared with the James splitting of $\Sigma\Omega\Sigma\mathbb{C}P^{2}$, implies that $\Sigma\mathbb{C}P^{2}\wedge\mathbb{C}P^{2}$ decomposes as the wedge of a three-cell complex and an $S^{7}$. But this cannot happen since the cohomology of this smash product is an indecomposable module over the mod-$2$ Steenrod algebra~\cite[Theorem 1.2]{SW}.

The inertness of the attaching map for the top cell of various families of Poincar\'{e} duality complexes has been shown. In particular, it holds for $(n-1)$-connected, $(2n)$-dimensional Poincar\'{e} duality complexes and $(n-1)$-connected, $(2n+1)$-dimensional Poincar\'{e} duality complexes under mild cohomological conditions \cite{BT2}. More generally, it has been shown for $(m-1)$-connected Poincar\'{e} duality complexes $M$  for which there is a retraction of $S^m$ off $M$ \cite{ST}. This retraction property has been shown to be closed under various operations on Poincar\'{e} duality complexes and has been used to decompose the based loop spaces of several large classes of Poincar\'{e} duality complexes~\cite{T1,ST}. 

In this paper we improve the result of F\'elix and Tanr\'e under certain conditions. For a graded $\mathbb{Q}$-module $V$ consisting of even degree elements, let $\mathrm{Sym}(V\otimes V)$ be the submodule of symmetric tensors in $V\otimes V$.

\begin{theorem}
\label{thm:localinert}
    Let $M$ be an $(m-1)$-connected, $n$-dimensional Poincar\'{e} duality complex such that $2 \leq m < n$, and let $f:S^{n-1} \to \overline{M}$ be the attaching map for the top cell. 
\begin{enumerate}
        \item If $m$ is odd and $H_m(M)$ contains a $\mathbb{Z}$-summand, then $f$ is inert when localised away from primes $p \leq \frac{n-m+3}{2}$.
        \item If $m=2m'$, let $S = \{2k \: | \: H_{2k}(M;\mathbb{Z}) \neq 0, H_{2i-1}(M;\mathbb{Z}) = 0, 1 \leq i \leq k,\: 2m \leq 2k \leq 4m-2\}$, and for $k \geq 1$, let $r_k = \mathrm{rank}(H_k(M;\mathbb{Z}))$.
     Suppose the following hold:
     \begin{enumerate}
         \item $r_m \geq 1$;
         \item $H_k(M;\mathbb{Z})$ is torsion-free for all $k \in S$; 
         \item if $\max(S)+1 \leq 4m-3$, then $H_{\max(S)+1}(M;\mathbb{Z})$ contains a $\mathbb{Z}$-summand;
         \item if $\max(S)+1 = 4m-1$, then either $H_{4m-1}(M;\mathbb{Z})$ contains a $\mathbb{Z}$-summand, or $H_{4m-1}(M;\mathbb{Z})$ does not contain a $\mathbb{Z}$-summand and the rational cup product map \[\mathrm{Sym}(H^{2m}(M;\mathbb{Q}) \otimes H^{2m}(M;\mathbb{Q})) \xrightarrow{\cup} H^{4m}(M;\mathbb{Q})\] is not an injection.
     \end{enumerate} Then $f$ is inert when localised away from primes \[p \leq \frac{n + \sum_{k \in S, k \neq \max(S)} r_{k}(k-1) + (r_{\max(S)}-1)(\max(S)-1)+3}{2} -\frac{\max(S)}{2}.\] \end{enumerate}
\end{theorem}

Some remarks should be made about Theorem \ref{thm:localinert}. The case when $m$ is odd was proved in \cite[Theorem 6.3]{T2}, the new result is the case when $m$ is even. The method for proving Theorem~\ref{thm:localinert} is completely different to that in~\cite{FT}, which used $p$-local models for spaces, analogous to rational models, that exist through a dimensional range. Theorem~\ref{thm:localinert}~(1) was proved in~\cite{T2} using explicit retractions. The approach to proving Theorem \ref{thm:localinert}~(2) is to define an auxiliary Poincar\'{e} duality complex $N$ that is of even connectivity, and whose inertness is equivalent to that of $M$. This allows us to apply Theorem \ref{thm:localinert}~(1). 


This approach is used in a different way in Section \ref{sec:6dim} to prove more than just inertness in a specific case. We decompose $\Omega M$ when $M$ is a simply-connected, $6$-dimensional Poincar\'{e} duality complex when $H_2(M;\mathbb{Z}) \cong \mathbb{Z} \oplus \mathbb{Z}$ and the auxiliary complex $N$ has admits a retraction of $S^{3}$. This retraction turns out to always hold after localising away from $2$ and $3$. With the additional assumptions that $H_{\ast}(M;\mathbb{Z})$ is torsion-free and $H_{3}(M;\mathbb{Z})\cong 0$, we give a refined decomposition of $\Omega M$ in which the factors depend only on spheres and Moore spaces.

Finally, we give two applications of Theorem~\ref{thm:localinert}. The first is to Moore's Conjecture. Combining inertness with a result of two of the authors~\cite{HT2} on torsion growth in homotopy groups, we show that after localising away from an explicit finite set of primes, $\Omega\overline{M}$ has the loop space of a wedge of two spheres as a retract. It follows that the $(n-1)$-skeleton $\overline{M}$ satisfies the hyperbolic form of Moore's Conjecture at all remaining primes, and is $\mathbb{Z}/p^r\mathbb{Z}$-hyperbolic for all $r\geq 1$. This gives further evidence for two conjectures on local hyperbolicity, namely \cite[Conjecture~1.6]{HT2} and \cite[Conjecture~1.7]{H2}, which are partial strengthenings of the hyperbolic direction of Moore's Conjecture. The second application concerns non-inertness. While inertness has received considerable attention, comparatively little work has been devoted to non-inertness and to the gaps between rational, local and integral inertness; see~\cite{H3}. Combining the Cohen--Neisendorfer construction~\cite{CN,GHMTW} with Burklund and Senger's subexponential growth theorem for spheres~\cite{BS}, we show that if \(\alpha\colon S^{b-1}\to S^a\) is \(p\)-local, \(3\leq a<b\) with $a,b$ odd and \(p>(b+3)/2\), then \(\alpha\) is not inert at \(p\). 

The paper is structured as follows. In Section \ref{sec:fibrations}, we compare the inert property for two Poincar\'{e} duality complexes related by a certain homotopy fibration. This is applied in Section \ref{sec:localinert} to prove Theorem \ref{thm:localinert}. Section \ref{sec:6dim} establishes the loop space decomposition of specific $6$-dimensional Poincar\'{e} duality complexes. Section~\ref{sec:mooreskeleton} applies Theorem~\ref{thm:localinert} to Moore's Conjecture for the $(n-1)$-skeleton $\overline{M}$, and Section~\ref{sec:CNnoninert} gives examples of non-inert maps between spheres. 
\smallskip 

\subsubsection*{Acknowledgements} The authors would like to thank the Heilbronn Institute for Mathematical Research for supporting this project through a Focused Research Group award. The first, second and third authors would like to thank the University of Southampton for its generous hospitality in supporting a research visit during which this work was discussed. The second author was supported by the Heilbronn Institute for Mathematical Research during preparation of this work. The third author was supported in part by the National Natural Science Foundation of China (Grant nos. 12331003 and 12288201) and the National Key R\&D Program of China (No. 2021YFA1002300). The fourth author was supported by EPSRC grant EP/Z534894/1.

\section{Inertness and fibrations}
\label{sec:fibrations}

Let $M$ be a simply-connected $n$-dimensional Poincar\'{e} duality complex. By Poincar\'{e} 
duality, $H_{n}(M;\mathbb{Z})\cong\mathbb{Z}$. As $M$ is simply-connected,  
it can be given a $CW$-structure with a single cell in dimension $n$. Let $\overline{M}$ be 
the $(n-1)$-skeleton of $M$. Then there is a homotopy cofibration 
\[\nameddright{S^{n-1}}{f}{\overline{M}}{i}{M}\] 
where $f$ is the attaching map for the $n$-cell and $i$ is the skeletal inclusion. 

Recall from the introduction that the attaching map $f$ is \emph{inert} if $\Omega i$ has a right homotopy inverse. In this section, we show that given a homotopy fibration of simply-connected Poincar\'{e} duality complexes, inertness of the attaching maps of the total space and base space are equivalent under certain conditions. For a given $(2m-1)$-connected, Poincar\'{e} duality complex $M$, we construct such a homotopy fibration localised away from a finite set of primes, in which $M$ is the base space. 

\begin{theorem} 
   \label{fibinert} 
   Suppose that there is a homotopy fibration 
   \(\nameddright{S^{2m-1}}{\alpha}{N}{}{M}\) 
   where $M$ and $N$ are Poincar\'{e} duality complexes, $M$ is $(2m-1)$-connected 
   and $n$-dimensional for $n>2m$, and $\alpha$ is null homotopic. Then the attaching map 
   for the top cell of~$N$ is inert if and only if the attaching map for the top cell of $M$ is inert. 
\end{theorem} 

\begin{proof} 
The first part of the proof sets the stage. Denote the map 
\(\namedright{N}{}{M}\) 
by $g$. Define the space~$\widehat{N}$ and the maps $\widehat{g}$ and $\widehat{i}$ by the 
homotopy pullback 
\[\diagram 
      \widehat{N}\rto^-{\widehat{g}}\dto^{\widehat{i}} & \overline{M}\dto^{i} \\ 
      N\rto^-{g} & M. 
 \enddiagram\]  
This induces a homotopy fibration diagram 
\begin{equation} 
  \label{Nhatdgrm} 
  \diagram 
      \Omega\overline{M}\rto^-{\widehat{\partial}}\dto^{\Omega i} & S^{2m-1}\rto^-{\widehat{\alpha}}\ddouble 
         & \widehat{N}\rto^-{\widehat{g}}\dto^{\widehat{i}} & \overline{M}\dto^{i} \\ 
      \Omega M\rto^-{\partial} & S^{2m-1}\rto^-{\alpha} & N\rto^-{g} & M 
   \enddiagram 
\end{equation} 
that defines the map $\widehat{\alpha}$ and the connecting maps $\widehat{\partial}$ and $\partial$. 
Since $M$ is $(2m-1)$-connected and $n$-dimensional for $n>2m$, the space $\overline{M}$ 
is not contractible. Therefore, by~\cite[Lemma 3.2]{CT}, the map $\widehat{\alpha}$ is null homotopic. This 
implies that $\widehat{\partial}$ has a right homotopy inverse 
\[\widehat{s}\colon\namedright{S^{2m-1}}{}{\Omega\overline{M}}.\] 
The homotopy commutativity of the left square in~(\ref{Nhatdgrm}) then implies that the composite 
\[s\colon\nameddright{S^{2m-1}}{\widehat{s}}{\Omega\overline{M}}{\Omega i}{\Omega M}\] 
is a right homotopy inverse for $\partial$. Consider the diagram 
\begin{equation} 
  \label{eehat} 
  \diagram 
     S^{2m-1}\times\Omega\widehat{N} 
          \rto^-{\widehat{s}\times\Omega\widehat{g}}\dto^{1\times\Omega\widehat{i}} 
       & \Omega\overline{M}\times\Omega\overline{M}\rto^-{\mu}\dto^{\Omega i\times\Omega i} 
       & \Omega\overline{M}\dto^{\Omega i} \\ 
     S^{2m-1}\times\Omega N\rto^-{s\times\Omega g} & \Omega M\times\Omega M\rto^-{\mu} 
       & \Omega M 
  \enddiagram 
\end{equation}  
where $\mu$ is the standard loop multiplication. The left square homotopy commutes by 
definition of~$s$ and the loops on the right square in~(\ref{Nhatdgrm}). The right square 
homotopy commutes since $\Omega i$ is an $H$-map. Let 
$\widehat{e}=\mu\circ(\widehat{s}\times\Omega\widehat{g})$ and $e=\mu\circ(s\times\Omega g)$ 
respectively be the composites along the top and bottom rows of this diagram. Observe that $\widehat{e}$ 
is a homotopy equivalence since $\widehat{s}$ is a right homotopy inverse for $\widehat{\partial}$, 
and $e$ is a homotopy equivalence since $s$ is a right homotopy inverse for $\partial$. 

Since $M$ is $n$-dimensional, the homotopy fibration 
\(\nameddright{S^{2m-1}}{\alpha}{N}{g}{M}\) 
implies that $N$ is $(n+2m-1)$-dimensional \cite{Q}. Let $\overline{N}$ be the $(n+2m-2)$-skeleton 
of $N$. Then there is a homotopy cofibration 
\[\nameddright{S^{n+2m-2}}{f'}{\overline{N}}{i'}{N}\] 
where $f'$ is the attaching map of the $(n+2m-1)$-cell of $N$ and $i'$ is the skeletal inclusion. 
By~\cite[Proposition 4.4]{CT}, there is a homotopy equivalence 
$\widehat{N}\simeq S^{n-1}\vee\overline{N}$, and there is a 
homotopy commutative square 
\begin{equation} 
  \label{NhatNbar} 
  \diagram 
     S^{n-1}\vee\overline{N}\rto^-{\simeq}\dto^{p_{2}} & \widehat{N}\dto^{\widehat{i}} \\ 
     \overline{N}\rto^-{i'} & N 
  \enddiagram 
\end{equation}  
where $p_{2}$ is the pinch map to the second wedge summand. Let $j$ be defined as the composite 
\(j\colon\nameddright{\overline{N}}{}{S^{n-1}\vee\overline{N}}{\simeq}{\widehat{N}}\), 
where the left map is the inclusion of the second wedge summand. Then~(\ref{NhatNbar}) 
implies that $\widehat{i}\circ j\simeq i'$. Let $r$ be the composite 
\(r\colon\nameddright{\widehat{N}}{\simeq}{S^{n-1}\vee\overline{N}}{p_{2}}{\overline{N}}\), 
where the left map is the inverse homotopy equivalence. Then~(\ref{NhatNbar}) implies 
that $\widehat{i}\simeq i'\circ r$. 

Suppose that the attaching map for the top cell of $N$ is inert. Then 
\(\namedright{\Omega\overline{N}}{\Omega i'}{\Omega N}\) 
has a right homotopy inverse 
\(t'\colon\namedright{\Omega N}{}{\Omega\overline{N}}\). 
Consider the diagram 
\[\diagram 
    S^{2m-1}\times\Omega N\rto^-{1\times t'}\drrdouble
       & S^{2m-1}\times\Omega\overline{N}\rto^-{1\times \Omega j}\drto^{1\times\Omega i'}    
       & S^{2m-1}\times\Omega\widehat{N}\rto^-{\widehat{e}}\dto^{1\times\Omega\widehat{i}}  
       & \Omega\overline{M}\dto^{\Omega i} \\  
     & & S^{2m-1}\times\Omega N\rto^-{e} & \Omega M. 
  \enddiagram\] 
The left triangle homotopy commutes since $t'$ is a right homotopy inverse for $\Omega i'$, 
the middle triangle homotopy commutes since $i'\simeq\widehat{i}\circ j$, and the right 
square homotopy commutes by~(\ref{eehat}) and the definitions of $\widehat{e}$ and $e$. 
Therefore the whole diagram homotopy commutes, implying that $e$ factors 
through $\Omega i$. As $e$ is a homotopy equivalence, this implies that $\Omega i$ has 
a right homotopy inverse. Hence the attaching map for the $n$-cell of $M$ is inert. 

Suppose that the attaching map for the top cell of $M$ is inert. Then 
\(\namedright{\Omega\overline{M}}{i}{\Omega M}\) 
has a right homotopy inverse 
\(t\colon\namedright{\Omega M}{}{\Omega\overline{M}}\). 
The loops of the homotopy pullback defining $\widehat{N}$ is itself a homotopy pullback, so there 
is an induced map $\widehat{t}$ to the homotopy pullback 
\[\xymatrix{ 
    \Omega N\ar[r]^-{\Omega g}\ar@{.>}[dr]^{\widehat{t}}\ar@/_/[ddr]_{=} & \Omega M\ar@/^/[dr]^{t} & \\ 
    & \Omega\widehat{N}\ar[r]^-{\Omega\widehat{g}}\ar[d]^{\Omega\widehat{i}} 
      & \Omega\overline{M}\ar[d]^{\Omega i} \\ 
   & \Omega N\ar[r]^-{\Omega g} & \Omega M. 
 }\] 
Combining the lower left triangle with the homotopy $\widehat{i}\simeq i'\circ r$, we obtain 
a homotopy commutative diagram 
\[\diagram 
    \Omega N\rto^-{\widehat{t}}\drdouble & \Omega\widehat{N}\rto^-{\Omega r}\dto^{\Omega\widehat{i}} 
        & \Omega\overline{N}\dto^{\Omega i'} \\ 
    & \Omega N\rdouble & \Omega N. 
  \enddiagram\] 
Thus $\Omega i'$ has a right homotopy inverse, implying that the attaching map for the top cell 
of $N$ is inert. 
\end{proof} 

\begin{remark} 
Theorem~\ref{fibinert} is a variation on a statement proved in~\cite[Theorem 1.6 (1)]{H}, where 
the hypotheses were different. In that case $M$ and $N$ were smooth orientable manifolds instead 
of Poincar\'{e} duality complexes, there was a spherical bundle 
\(\nameddright{S^{2m-1}}{}{N}{}{M}\) 
instead of a homotopy fibration, and the value of $m$ was restricted to $m\in\{1,2,4\}$. Note that 
this restriction on~$m$ was motivated by the fact that $g$ being null homotopic implies that 
$S^{2m-1}$ retracts off $\Omega M$ and so is an $H$-space. Localised away from $2$, any odd 
dimensional sphere is an $H$-space, so Theorem~\ref{fibinert} applies for any positive integer $m$. In all these ways, Theorem~\ref{fibinert} generalises~\cite[Theorem 1.6~(1)]{H}. 
On the other hand, in that case $M$ need not have been $(2m-1)$-connected, so this 
hypothesis in our case is weaker. 
\end{remark} 

Next, given a $(2m-1)$-connected, $n$-dimensional Poincar\'{e} duality complex $M$ with $n > 2m$, we show that we can construct a homotopy fibration satisfying the hypotheses of Theorem \ref{fibinert} localised away from a finite set of primes, in which $M$ is the base space. In particular, we will prove the following. 

\begin{proposition} 
   \label{localfib} 
   Let $M$ be a $(2m-1)$-connected, $n$-dimensional Poincar\'{e} duality complex with a class 
   $x\in H_{2m}(M;\mathbb{Z})$ that generates a primitive $\mathbb{Z}$-summand. 
   Then after localizing away from primes $p\leq\frac{n+3}{2}-m$, there is a homotopy fibration 
   \(\nameddright{S^{2m-1}}{\alpha}{N}{g}{M}\) 
   where \begin{enumerate}
       \item $x$ is in the cokernel of $g_{\ast}$,
       \item $\alpha$ is null homotopic,
       \item $N$ is a $(2m-1)$-connected, $(n+2m-1)$-dimensional Poincar\'{e} duality complex,
       \item $\mathrm{rank}(H_{2m}(N)) = \mathrm{rank}(H_{2m}(M))-1$,
   \end{enumerate}
\end{proposition}

The proof of Proposition \ref{localfib} will proceed in a few steps. Let $M$ be a $(2m-1)$-connected $n$-dimensional Poincar\'{e} duality complex with a class $x\in H_{2m}(M;\mathbb{Z})$ that generates a primitive $\mathbb{Z}$-summand. The class $y \in H^{2m}(M;\mathbb{Z})$ dual to $x$ is represented by a map $r:M \to K(\mathbb{Z},2m)$. We first show a factorisation of $r$ through $\Omega S^{2m+1}$ after localisation away from a finite set of primes. 

\begin{lemma}
\label{lem:liftthroughloop}
    Localise away from primes $p\leq\frac{n+3}{2}-m$. Then there exists a map $r':M \to \Omega S^{2m+1}$ and a homotopy commutative diagram \[\begin{tikzcd}
	& {\Omega S^{2m+1}} \\
	M & {K(\mathbb{Z},2m),}
	\arrow["s", from=1-2, to=2-2]
	\arrow["{r'}", from=2-1, to=1-2]
	\arrow["r", from=2-1, to=2-2]
\end{tikzcd}\] where $s$ represents a generator $z \in H^{2m}(\Omega S^{2m+1};\mathbb{Z})$. Moreover, $(r')^{\ast}(z)=y$. 
\end{lemma}
\begin{proof}
Note there is only one $\mathbb{Z}$-summand in $\pi_{\ast}(\Omega S^{2m+1})$, 
occurring in dimension $2m$. For a prime $p$, the least dimensional nontrivial $p$-torsion homotopy group in $\pi_{\ast}(\Omega S^{2m+1})$ occurs in dimension $2m+2p-3$. Therefore, after localizing 
away from primes $p\leq\frac{n+3}{2}-m$, the map $s$ induces an isomorphism on homotopy 
groups in dimensions~$\leq n$, and is therefore a homotopy equivalence in that range. Thus 
as $M$ is $n$-dimensional, any map 
\(\namedright{M}{}{K(\mathbb{Z},2m)}\) 
lifts through $s$ to a map 
\(\namedright{M}{}{\Omega S^{2m+1}}\). 

In particular, let $y\in H^{2m}(M;\mathbb{Z})$ be dual to $x$ and let 
\(r\colon\namedright{M}{}{K(\mathbb{Z},2m)}\) 
represent $y$. Localise away from primes $p\leq\frac{n+3}{2}-m$. Then there is a lift 
\[\diagram 
      & \Omega S^{2m+1}\dto^{s} \\ 
      M\rto^-(0.6){r}\urto^{r'} & K(\mathbb{Z},2m) 
  \enddiagram\] 
for some map $r'$. As $r$ represents $y$ and $s$ represents $z$, we have 
$(r')^{\ast}(z)=y$. 
\end{proof}

Continuing, let $r'$ be defined as in Lemma \ref{lem:liftthroughloop}, and define the space $N'$ and the maps $g'$ and $\alpha'$ by the homotopy fibration sequence 
\begin{equation} 
  \label{loop2fib} 
  \namedddright{\Omega^{2} S^{2m+1}}{\alpha'}{N'}{g'}{M}{r'}{\Omega S^{2m+1}}. 
\end{equation} 
Lemma \ref{lem:liftthroughloop} implies $(r')^{\ast}(z)=y$, and so $(g')^{\ast}(y)=0$. This implies that in homology, $x$ is 
in the cokernel of $g'_{\ast}$. We now modify the homotopy fibration~(\ref{loop2fib}). Let 
\[E^{2}\colon\namedright{S^{2m-1}}{}{\Omega^{2} S^{2m+1}}\] 
be the double suspension, which is the double adjoint of the identity map on $S^{2m+1}$. 
Let $W_{m}$ be the homotopy fibre of $E^{2}$. It is well known that $W_{m}$ is 
rationally trivial and, when localised at a prime~$p$, that it is $(2mp-4)$-connected. Therefore, localised away from primes $p\leq\frac{n+3}{2}-m$, it follows that $W_{m}$ is 
at least $(n+2m-1)$-connected. 
Let $N$ be the $(n+2m-1)$-skeleton of $N'$ and let $g$ be the composite 
\[g\colon\nameddright{N}{}{N'}{g'}{M}\] 
where the left map is the skeletal inclusion. As $x$ is in the cokernel of $g'_{\ast}$, 
it follows that $x$ is in the cokernel of $g_{\ast}$. 
Define the spaces $F$ and $G$, and the maps $\epsilon$ 
and $\alpha$, by the homotopy fibration diagram 
\begin{equation} 
  \label{Fepsilondgrm} 
  \diagram 
     G\rdouble\dto & G\dto & \\ 
     F\rto^-{\alpha}\dto^{\epsilon} & N\rto^-{g}\dto & M\ddouble \\ 
     \Omega^{2} S^{2m+1}\rto^-{\alpha'} & N'\rto^-{g'} & M. 
  \enddiagram 
\end{equation} We show that the middle row is the homotopy fibration we require in Proposition \ref{localfib}.

\begin{proof}[Proof of Proposition \ref{localfib}]
    We will show that $F\simeq S^{2m-1}$ in \eqref{Fepsilondgrm} localised away from the primes $p\leq\frac{n+3}{2}-m$. Take homology with coefficients in $\mathbb{Z}$ localised 
away from primes $p\leq\frac{n+3}{2}-m$. First consider the Serre spectral sequence for the 
homotopy fibration 
\(\nameddright{\Omega^{2} S^{2m+1}}{\alpha'}{N'}{g'}{M}\) 
that converges to $H_{\ast}(N')$. For all bidegrees $(p,q)$ there is a module isomorphism 
\[E^{2}_{p,q}(N')\cong H_{p}(M)\otimes H_{q}(\Omega^{2} S^{2m+1}).\] 
Consider bidegrees of the form $(p,q)$ for $q\leq n+2m$. As $W_{m}$ is at least $(n+2m)$-connected, 
the restriction of 
\(\namedright{S^{2m-1}}{E^{2}}{\Omega^{2} S^{2m+1}}\) 
to $(n+2m)$-skeletons is a homotopy equivalence. Thus in bidegrees $(p,q)$ for $q\leq n+2m$, 
the $E^{2}$-page is concentrated on two horizontal 
lines in bidegrees $(p,0)$ and $(p,2m-1)$. As $M$ is $n$-dimensional, the only 
nontrivial groups occur for $p\leq n$, implying that the only possible nontrivial differentials are 
$d^{2},\ldots,d^{n}$. Applying one of these to an element of bidegree $(p,2m-1)$ gives 
an element in bidegree $(p',t)$ for some $p'$ and $2m-1<t< n+2m$, but for those values  
of $t$ all such bidegrees are zero. Thus all differentials are zero on bidegrees $(p,2m-1)$. 
On the other hand, all the elements of bidegree $(p,0)$ have $p\leq n$ 
so differentials on such elements have image in bidegrees of the form 
$(p',q)$ for some $p'$ and $0<q\leq n-1$. These groups are zero except for those of the 
form $(p',2m-1)$, implying that the only possible nontrivial differential is~$d^{2m}$. Hence 
if the spectral sequence converging to $H_{\ast}(N')$ is restricted to  
bidegrees~$(p,q)$ for $q\leq n+2m$ then the only possible nonzero differential is $d^{2m}$ 
in bidegrees $(p,0)$. 

Notice that this argument only depended on $F$ having the same homology as $S^{2m-1}$ 
in dimensions~$\leq n+2m$, so if $F$ had the same homology as $S^{2m-1}$ in dimensions~$< t$ 
for some $t>2n+m$ then the restriction of the spectral sequence to bidegrees $(p,q)$ for $q< t$ 
would have nonzero elements only in bidegrees $(p,0)$ and $(p,2m-1)$ and the only possible 
nonzero differential would be $d^{2m}$ on elements in bidegrees $(p,0)$.   

One consequence of the spectral sequence calculation is that the $(n+2m-1)$-skeleton $N$ of~$N'$ 
has the same homology as the $(n+2m)$-skeleton, so is homotopy equivalent to it. With $N$ 
being the $(n+2m)$-skeleton of $N$, the space $G$ in~(\ref{Fepsilondgrm}) is $(n+2m-1)$-connected, 
implying that the map~$\epsilon$ is a homotopy equivalence in dimensions~$\leq n+2m-1$. 
In particular, the homotopy fibration 
\(\nameddright{F}{\alpha}{N}{g}{M}\) 
has the property that if the Serre spectral sequence converging to $H_{\ast}(N)$ is restricted to  
bidegrees~$(p,q)$ for $q\leq n+2m-1$ then the only possible nonzero differential is $d^{2m}$ 
in bidegrees $(p,0)$. 

We now show that $F\simeq S^{2m-1}$ in all dimensions. If not, then there must be a 
homology class $z\in H_{t}(F)$ for some $t>n+2m-1$ that is of least degree. Consider 
the spectral sequence for the homotopy fibration 
\(\nameddright{F}{\alpha}{N}{g}{M}\) 
converging to $H_{\ast}(N)$. On the $E^{2}$-page, $z$ is in bidegree $(0,t)$. For degree 
reasons, all differentials are zero on $z$. If $z$ is in the image of a differential then $z=d^{r}(w)$ 
for some $w$ in bidegree $(p,q)$ for $q<t$. The conclusion of the previous paragraph implies 
that $q$ is either $0$ or $2m-1$ and the only possible nonzero differential on $w$ is $d^{2m}$. 
But as $p\leq n$, this implies that $z$ has total degree~$\leq n+2m-1$, whereas $z$ has 
total degree $t>n+2m-1$, a contradiction. Therefore $z$ is not in the image of a differential. 
Hence $z$ must survive the spectral sequence to give an element in $E^{\infty}$, and for 
degree reasons there can be no extension problems, so $z$ survives to an element in 
$H_{t}(N)$. But as $N$ is the $(n+2m-1)$-skeleton of $N$, it is at most $(n+2m-1)$-dimensional, 
implying that $H_{t}(N)\cong 0$ as $t>n+2m-1$. This contradiction implies that~$z$ cannot 
exist, that is, $H_{\ast}(F)\cong H_{\ast}(S^{2m-1})$ in all degrees. Thus the inclusion 
\(\namedright{S^{2m-1}}{}{F}\) 
of the bottom cell induces an isomorphism in homology in all degrees, so as spaces are 
simply-connected, this inclusion must be a homotopy equivalence by Whitehead's Theorem. 

Consequently, as $F$ is defined as the homotopy fibre of $g$ in~(\ref{Fepsilondgrm}), 
we obtain a homotopy fibration 
\(\nameddright{S^{2m-1}}{}{N}{g}{M}\). 
Finally, the Serre spectral sequence converging to $H_{\ast}(N)$ has $E^{2}$-page equal to zero in 
total degrees~$\leq 2m-1$, which implies that $N$ is $(2m-1)$-connected. Thus the map 
\(\namedright{S^{2m-1}}{}{N}\) 
is null homotopic. 

The fact that $N$ is a Poincar\'{e} duality complex follows from \cite{Q}. The assertion about homology follows from the fact that $x$ is in the cokernel of $g_*$.
\end{proof}

\begin{remark}
    If $M$ is simply-connected, then there is an improvement. Let $x \in H_2(M;\mathbb{Z})$ generate a $\mathbb{Z}$-summand, and let $y \in H^2(M;\mathbb{Z})$ be the dual of $x$. Then $y$ is represented by a map $r:M \to BS^1$. Integrally, we obtain a homotopy fibration sequence \[\Omega M \xrightarrow{\partial}S^1 \xrightarrow{\alpha} N \to M \to BS^1.\]

    By the Hurewicz isomorphism $x$ corresponds to a generator of $\pi_1(\Omega M)$. In particular, there is a map $S^1 \to \Omega M$ such that the composite $S^1 \to \Omega M \xrightarrow{\partial} S^1$ is a homotopy equivalence. Hence, $\partial$ has a right homotopy inverse, and so $\alpha$ is null homotopic. Theorem \ref{fibinert} then implies the attaching map for the top cell of $N$ is inert if and only if the attaching map for the top cell of $M$ is inert.
\end{remark}

\section{Local inertness}
\label{sec:localinert}

In this section, we prove Theorem \ref{thm:localinert}. 
Let $M$ be an $(m-1)$-connected, $n$-dimensional Poincar\'{e} duality complex with $m \geq 1$. When $m$ is odd, Theorem \ref{thm:localinert} was proved in \cite[Theorem 6.3]{T2}.

\begin{proposition}
\label{prop:evenconnectivitycase}
    Let $M$ be an $(m-1)$-connected, $n$ dimensional Poincar\'{e} duality complex with $m$ odd. If $H_m(M)$ has a $\mathbb{Z}$-summand, then localised away from primes $p \leq \frac{n-m+3}{2}$, the attaching map for the top cell is inert. \qed 
\end{proposition}

The remainder of this section will prove Theorem \ref{thm:localinert} when $m$ is even. Let $f_M$ be the attaching map for the top cell of $M$. The strategy will be to iterate Theorem \ref{fibinert} and Proposition \ref{localfib} to construct a Poincar\'{e} duality complex $N$ of even connectivity, such that the inertness of the attaching map for the top cell $f_N$ of $N$ is equivalent to that of $f_M$. The fact that $N$ is of even connectivity implies that Proposition \ref{prop:evenconnectivitycase} can be applied to obtain inertness of  $f_N$, and hence $f_M$.

We first recall some notation from the introduction. Let $M$ be a $(2m-1)$-connected, $n$-dimensional Poincar\'{e} duality complex with $m \geq 1$. Let \[S = \{2k \: | \: H_{2k}(M;\mathbb{Z}) \neq 0, H_{2i-1}(M;\mathbb{Z}) = 0, 1 \leq i \leq k,\: 2m \leq 2k \leq 4m-2\}.\] For $k \geq 1$, let $r_k = \mathrm{rank}(H_k(M;\mathbb{Z}))$. Let $\mathcal{P}$ be the set of primes $p$ satisfying the inequality \[p \leq \frac{n + \sum_{k \in S, k \neq \max(S)} r_{k}(k-1) + (r_{\max(S)}-1)(\max(S)-1)+3}{2} -\frac{\max(S)}{2}.\]

\begin{lemma}
\label{lem:ReduceN}
    Let $M$ be a $(2m-1)$-connected, $n$-dimensional Poincar\'{e} duality complex, such that $r_{2m} \geq 1$. Suppose that for all $k \in S$, $H_k(M;\mathbb{Z})$ is torsion-free. Then 
    \begin{enumerate}
        \item[(a)] localised away from primes in $\mathcal{P}$, there exists a $(\max(S))$-connected Poincar\'{e} duality complex $N$ of dimension $n + \sum_{k \in S} r_{k}(k-1)$ such that the attaching map for the top cell of $N$ is inert if and only if the attaching map for the top cell of $M$ is inert; 
        \item[(b)] there is an isomorphism $\pi_{\max(S)+1}(M) \otimes \mathbb{Q} \cong \pi_{\max(S)+1}(N) \otimes \mathbb{Q}$; 
        \item[(c)] if we have that $\max(S)+1 \leq 4m-3$ then $H_{\max(S)+1}(N) \cong H_{\max(S)+1}(M)$.
    \end{enumerate}
\end{lemma}
\begin{proof}
Let $S = \{k_1,\cdots,k_\ell\}$, with $k_i < k_{i+1}$ for $1 \leq i \leq \ell-1$. Note that since $r_{2m} \geq 1$, $k_1 = 2m$. Localised away from $\mathcal{P}$, we define a sequence of Poincar\'{e} duality complexes \[M = N_0,N_1,\cdots,N_{r_{k_1}},N_{r_{k_1}+1},\cdots,N_{r_{k_1}+r_{k_2}},\cdots,N_{\sum_{i=1}^\ell r_{k_i}},\] such that the attaching map for the top cell of $N_i$ is inert if and only if the attaching map for the top cell of $N_{i+1}$ is inert, and $N_{\sum_{i=1}^\ell r_{k_i}}$ is the Poincar\'{e} duality complex $N$ in the statement of the lemma. 

First, we define $N_1$. Since $r_{2m} \geq 1$, there exists $x \in H_{2m}(M)$ which generates a primitive $\mathbb{Z}$-summand. Proposition \ref{localfib} implies that localised away from primes $p \leq \frac{n+3}{2}-m$ there is a homotopy fibration \[S^{2m-1} \xrightarrow{\alpha_1} N_1 \rightarrow M,\] where $N_1$ is a Poincar\'{e} duality complex of dimension $n+2m-1$, $\mathrm{rank}(H_{2m}(N_1)) = r_{2m}-1$ and $\alpha_1$ is null homotopic. By Theorem \ref{fibinert}, $N_1$ is inert if and only if $M$ is inert. There are three cases.

If $\ell = 1$, and $r_{2m} = 1$, take $N=N_{1}$ and we are done. If $r_{2m}-1 \geq 2$, apply Proposition \ref{localfib} again to obtain a homotopy fibration \[S^{2m-1} \xrightarrow{\alpha_2} N_2 \rightarrow N_1\] localised away from primes $p \leq \frac{n+(2m-1)+3}{2}-m$, where $N_2$ is a Poincar\'{e} duality complex of dimension $n+2(2m-1)$, where $\mathrm{rank}(H_{2m}(N_2)) = r_{2m}-2$ and $\alpha_2$ is null homotopic. This argument can be iterated for each generator of a $\mathbb{Z}$-summand in $H_{2m}(M)$ so that, localised away from primes $p \leq \frac{n+(r_{2m}-1)(2m-1)+3}{2}-m$, we obtain a Poincar\'{e} duality complex $N_{r_{k_1}}$ of dimension $n+r_{k_1}(2m-1)$ with $H_{2m}(N_{r_{k_1}})=0$. In particular, $N_{r_{k_1}}$ is $k_1$-connected. If $\ell=1$, take $N=N_{r_{k_1}}$ and we are done. 

If $\ell \geq 2$ then this argument is repeated for each $\mathbb{Z}$-summand in $H_{k_2}(N_{r_{k_1}}) \cong H_{k_2}(M)$ so that, localised away from primes $p \leq \frac{n+r_{2m}(2m-1)+(r_{k_2}(k_2-1))+3}{2}-\frac{k_1}{2}$, we obtain a Poincar\'{e} duality complex $N_{r_{k_1}+r_{k_2}}$ of dimension $n+r_{k_1}(k_1-1) + r_{k_2}(k_2-1)$ which is $k_2$-connected. Repeating this argument for each $k_i$ shows that, localised away from primes in $\mathcal{P}$, $N_{\sum_{i=1}^\ell r_{k_i}}$ is the Poincar\'{e} duality complex $N$ in the statement of the lemma, proving part (a). 

For $1 \leq j \leq \sum_{i=1}^{\ell} r_{k_i}$, there is a homotopy fibration $S^{\mathrm{Conn}(N_{j-1})} \xrightarrow{\alpha_j} N_j \to N_{j-1}$, where $\mathrm{Conn}(N_{j-1})$ denotes the connectivity of $N_{j-1}$, and $\alpha_j$ is null homotopic. By construction, we have $2m-1 \leq \mathrm{Conn}(N_{j-1}) \leq \max(S)-1$. As $\mathrm{Conn}(N_{j-1})$ is odd, $\pi_*(S^{\mathrm{Conn}(N_{j-1})}) \otimes \mathbb{Q}$ is trivial in all degrees except $\mathrm{Conn}(N_{j-1})$. Therefore, the long exact sequence of homotopy groups implies that $\pi_{\max(S)+1}(N_j)\otimes\mathbb{Q} \cong \pi_{\max(S)+1}(N_{j-1})\otimes\mathbb{Q}$. Since this holds for every $j$, we obtain an isomorphism $\pi_{\max(S)+1}(M)\otimes\mathbb{Q} \cong \pi_{\max(S)+1}(N)\otimes\mathbb{Q}$. This proves (b).

As for (c), the Blakers-Massey theorem implies that the homotopy fibration defining $N_j$ is a homotopy cofibration in dimensions $\leq 2\mathrm{Conn}(N_{j-1})$. Since $2m-1 \leq \mathrm{Conn}(N_{j-1})$, we obtain that the homotopy fibration is a homotopy cofibration in dimensions $\leq 4m-2$. If $\max(S)+1 \leq 4m-3$, then as $\mathrm{Conn}(N_{j-1})\leq \max(S)-1$, the long exact sequence of homology groups implies that $H_{\max(S)+1}(N_j) \cong H_{\max(S)+1}(N_{j-1})$. Since this holds for all $j$, $H_{\max(S)+1}(N) \cong H_{\max(S)+1}(M)$.
\end{proof}

Using Lemma \ref{lem:ReduceN}, inertness for the attaching map for the top cell of $M$ can be reduced to the corresponding problem for the more highly connected $N$. To show that the attaching map for the top cell of $N$ is inert, we wish to apply Proposition \ref{prop:evenconnectivitycase}. However, if $N$ is $(k-1)$-connected, this requires that $H_k(N)$ contains a $\mathbb{Z}$-summand. The next result classifies when this occurs. 

\begin{lemma}
\label{lem:NZsummand}
    Let $M$ be a $(2m-1)$-connected, $n$-dimensional Poincar\'{e} duality complex, such that $r_{2m} \geq 1$. Suppose that for all $k \in S$, $H_k(M;\mathbb{Z})$ is torsion-free. Let $N$ be as in Lemma \ref{lem:ReduceN}. There are two cases. 
\begin{itemize} 
   \item[(a)] If $\max(S)+1 \leq 4m-3$, then $H_{\max(S)+1}(N)$ contains a $\mathbb{Z}$-summand if and only if $H_{\max(S)+1}(M)$ contains a $\mathbb{Z}$-summand. 
    \item[(b)] If $\max(S)+1 = 4m-1$, then $H_{4m-1}(N)$ contains a $\mathbb{Z}$-summand if and only if one of the following conditions hold: \begin{enumerate}
        \item $H_{4m-1}(M)$ contains a $\mathbb{Z}$-summand,
        \item $H_{4m-1}(M)$ does not contain a $\mathbb{Z}$-summand and the rational cup product map 
        \[\mathrm{Sym}(H^{2m}(M;\mathbb{Q}) \otimes H^{2m}(M;\mathbb{Q})) \xrightarrow{\cup} H^{4m}(M;\mathbb{Q})\] is not an injection. 
    \end{enumerate} 
\end{itemize} 
\end{lemma}
\begin{proof}
    First, suppose $\max(S)+1 \leq 4m-3$. Then Lemma \ref{lem:ReduceN}~(c) implies that there is an isomorphism $H_{\max(S)+1}(N) \cong H_{\max(S)+1}(M)$, so part~(a) holds.

    Now suppose $\max(S)+1 = 4m-1$. The Hurewicz isomorphism implies $H_{4m-1}(N) \cong \pi_{4m-1}(N)$ and Lemma \ref{lem:ReduceN}~(b) implies that $\pi_{4m-1}(N) \otimes \mathbb{Q} \cong \pi_{4m-1}(M) \otimes \mathbb{Q}$. Hence, it suffices to show that $\pi_{4m-1}(M) \otimes \mathbb{Q}$ is non-trivial.

    If $H_{4m-1}(M)$ contains a $\mathbb{Z}$-summand, it is well known that the rational Hurewicz homomorphism is a surjection in this dimension, in which case $\pi_{4m-1}(M) \otimes \mathbb{Q}$ is non-trivial and we are done. If $H_{4m-1}(M)$ does not contain a $\mathbb{Z}$-summand, then by \cite{L}, there is an exact sequence \begin{equation}\label{eqn:pi_3exact}0 \rightarrow (\pi_{4m-1}(M) \otimes \mathbb{Q})^* \rightarrow \mathrm{Sym}(H^{2m}(M;\mathbb{Q}) \otimes H^{2m}(M;\mathbb{Q})) \xrightarrow{\cup} H^{4m}(M;\mathbb{Q}),\end{equation}  where $(\pi_{4m-1}(M)\otimes\mathbb{Q})^{\ast}$ is the hom-dual of $\pi_{4m-1}(M)\otimes\mathbb{Q}$. In particular, $(\pi_{4m-1}(M) \otimes \mathbb{Q})^*$ is isomorphic to the kernel of the cup product map. Hence, $(\pi_{4m-1}(M) \otimes \mathbb{Q})^*$ is trivial if and only if the cup product map is an injection. Dualising, $\pi_{4m-1}(M) \otimes \mathbb{Q}$ is non-trivial and we are done.
\end{proof}

With Lemmas \ref{lem:ReduceN} and \ref{lem:NZsummand}, we can prove Theorem \ref{thm:localinert}.

\begin{proof}[Proof of Theorem \ref{thm:localinert}]
The case where $m$ is odd is Proposition \ref{prop:evenconnectivitycase}, so suppose that $M$ is $(2m-1)$-connected and satisfies the hypotheses in the statement of Theorem \ref{thm:localinert}.

Localise away from primes in $\mathcal{P}$. By Lemma \ref{lem:ReduceN}~(a), there exists a $(\max(S))$-connected Poincar\'{e} duality complex $N$ of dimension $n + \sum_{k \in S} r_{k}(k-1)$ such that the attaching map for the top cell of $N$ is inert if and only if the attaching map for the top cell of $M$ is inert. Since $\mathrm{max}(S)$ is even, Proposition \ref{prop:evenconnectivitycase} implies that the attaching map for the top cell of $N$ is inert if $H_{\max(S)+1}(N)$ contains a $\mathbb{Z}$-summand. However, this follows from Lemma \ref{lem:NZsummand}. 
\end{proof}

More can be said in the special case of simply-connected $6$-dimensional Poincar\'{e} duality complexes with $\mbox{rank}(H_{2}(M;\mathbb{Z}))\geq 1$. 

\begin{corollary}
\label{cor:6dim}
    Let $M$ be a simply-connected, $6$-dimensional Poincar\'{e} duality complex and suppose that $H_2(M;\mathbb{Z}) \cong \oplus_{i=1}^r \mathbb{Z}$ for some integer $r\geq 1$. Then the attaching map of the $6$-cell of $M$ is inert when localised away from $p \leq 3+\frac{r}{2}$ if and only if $M \not \simeq_\mathbb{Q} \mathbb{C}P^3$. 
\end{corollary}
\begin{proof}
The Universal Coefficient Theorem and Poincar\'{e} duality implies \[H_2(M;\mathbb{Q}) \cong H^2(M;\mathbb{Q}) \cong H^4(M;\mathbb{Q}).\] Therefore, if $r \geq 2$, the cup product map is not injective for dimensional reasons: there are $\binom{r}{2}$ cup products of degree $2$ generators but only $r$ degree $4$ generators. Hence, Theorem \ref{thm:localinert} implies that the attaching map of the $6$-cell of $M$ is inert localised away from $p \leq 3+\frac{r}{2}$.

If $r=1$, let $x\in H^{2}(M;\mathbb{Q})$, $y\in H^{4}(M;\mathbb{Q})$ and $z\in H^{6}(M;\mathbb{Q})$ be the generators. Then $x^{2}=ty$ for some $t\in\mathbb{Q}$ and $xy=z$. If $t=0$ then $M \simeq_{\mathbb{Q}} S^2 \times S^4$ and if $t\neq 0$ then $M \simeq_{\mathbb{Q}} \mathbb{C}P^3$. If $M \simeq_{\mathbb{Q}} S^2 \times S^4$, the attaching map for the $6$-cell is inert localised away from $p \leq 3+\frac{r}{2}$ since the rational cup product map is not injective. If $M\simeq_{\mathbb{Q}}\mathbb{C}P^{3}$ then $\pi_{7}(M)$ is rationally nontrivial while $\overline{M}\simeq_{\mathbb{Q}}\mathbb{C}P^{2}$ implies that $\pi_{7}(\overline{M})$ is rationally trivial. Thus the inclusion $\overline{M}\rightarrow M$ does not have a right homotopy inverse, implying that the atttaching map for the $6$-cell of $M$ is not inert.
\end{proof}

\section{A loop space decomposition of certain simply-connected $6$-dimensional Poincar\'{e} duality complexes}
\label{sec:6dim}

In this section, we prove a loop space decomposition for simply-connected $6$-dimensional Poincar\'{e} duality complexes whose second cohomology is torsion-free and of rank $2$, and the third cohomology group is trivial. We begin with an initial decomposition. 

\begin{lemma}
\label{lem:initialdecomp6dim}
    Let $M$ be a simply-connected $6$-dimensional Poincar\'{e} duality complex. Assume that $H^2(M;\mathbb{Z}) \cong \mathbb{Z} \oplus \mathbb{Z}$. Then there exists a homotopy fibration $S^1 \times S^1 \xrightarrow{\alpha} N \to M$ where $N$ is a $2$-connected, $8$-dimensional Poincar\'{e} duality complex and $\alpha$ is null homotopic. Moreover, $H_3(N)$ contains a $\mathbb{Z}$-summand and there is a homotopy equivalence \[\Omega M \simeq S^1 \times S^1 \times \Omega N.\]
\end{lemma}

\begin{proof}
    Let $y_1,y_2 \in H^2(M;\mathbb{Z})$ be the generators of the $\mathbb{Z}$ summands. For $i \in \{1,2\}$, $y_i$ is represented by a map $r_i:M \to K(\mathbb{Z},2)$. Write $K(\mathbb{Z},2)$ as $\mathbb{C}P^{\infty}$ and note that 
$\Omega\mathbb{C}P^{\infty}\simeq S^{1}$. Let 
\(s\colon\namedright{M}{}{\mathbb{C}P^\infty \times \mathbb{C}P^\infty}\) 
be the product of the maps $r_{i}$. Define the space $N$ by the homotopy fibration sequence 
\begin{equation} 
  \label{Nfib} 
  \namedddright{S^1 \times S^1}{\alpha}{N}{}{M}{s}{\mathbb{C}P^\infty \times \mathbb{C}P^\infty}. 
\end{equation}  

Since $N$ fits in a homotopy fibration where the base and fibre are Poincar\'{e} duality complexes, 
by~\cite{Q} the space $N$ is also a Poincar\'{e} duality complex, and it has dimension $8$. By 
the Hurewicz isomorphism, $\pi_{2}(M)\cong H_{2}(M;\mathbb{Z})$, 
so as $s$ represents both degree $2$ generators, it induces an isomorphism 
on $\pi_{2}$. Therefore $\pi_{2}(N)\cong 0$, so $N$ is at least $2$-connected. 

The long exact sequence of homotopy groups implies that $\pi_3(M) \otimes \mathbb{Q} \cong \pi_3(N) \otimes \mathbb{Q}$. Hence, to show that $H_3(N)$ contains a $\mathbb{Z}$-summand, it suffices to show that $\pi_3(M)$ contains a $\mathbb{Z}$-summand. Arguing as in the proof of Lemma \ref{lem:NZsummand}, if $H_3(M)$ contains a $\mathbb{Z}$-summand, then surjectivity of the Hurewicz homomorphism implies $\pi_3(M)$ contains a $\mathbb{Z}$-summand. 

Now suppose $H_3(M)$ does not contain a $\mathbb{Z}$ summand. By Poincar\'{e} duality, $H^2(M;\mathbb{Q}) \cong H^4(M;\mathbb{Q})$. Since the rank of $H_2(M;\mathbb{Z})$ is 2, the rational cup product map is not injective for dimensional reasons: there are $3$ choices of cup products (up to ordering) of degree $2$ generators while there are only $2$ degree $4$ generators. Hence, \cite{L} implies that $\pi_3(M)$ contains a $\mathbb{Z}$-summand. 

To show the loop space decomposition, it suffices to show that $\alpha$ is null homotopic. Consider the map $\Omega M\xrightarrow{\Omega s} S^{1}\times S^{1}$ obtained by extending \eqref{Nfib} to the left. Since $s$ induces an isomorphism on $\pi_2$, $\Omega s$ induces an isomorphism on $\pi_{1}$. Therefore there are maps $a,b\colon S^{1}\rightarrow\Omega M$ such that $\Omega s\circ a$ and $\Omega s\circ b$ are the inclusions of the first and second wedge summands respectively. As $\Omega M$ is an $H$-space, $a$ and $b$ can be multiplied together, and as $\Omega s$ is an $H$-map, the composite $\Omega s\circ (a\cdot b)$ is a homotopy equivalence. Hence, $\Omega s$ has a right homotopy inverse, implying that $\alpha$ is null homotopic.
\end{proof}

Lemma \ref{lem:initialdecomp6dim} implies that to give a loop space decomposition for $M$, it suffices to give a decomposition for $N$. Notice that Lemma~\ref{lem:initialdecomp6dim} is easily generalised to $H^{2}(M;\mathbb{Z})\cong\oplus_{i=1}^{r}\mathbb{Z}$ for any $r\geq 2$, but with $N$ replaced by a $2$-connected Poincar\'{e} Duality complex $N_{r}$ of dimension $6+r$. The $r=2$ case is special because the homotopy type of $\Omega N$ can be identified through work of~\cite{ST}.

Let $\overline{N}$ be the $7$-skeleton of $N$. Since $N$ is a $2$-connected, $8$-dimensional Poincar\'{e} duality complex, there is a homotopy cofibration \[S^7 \xrightarrow{f_N} \overline{N} \xrightarrow{i} N,\] where $f_N$ is the attaching map of the $8$-cell. 
By Poincar\'{e} duality, 
\begin{equation} 
  \label{8dimcohomology} 
  H_m(N;\mathbb{Z})\cong\left\{\begin{array}{ll} 
    \mathbb{Z} & \mbox{if $m\in\{0,8\}$} \\ 
    \mathbb{Z}^{d}\oplus T & \mbox{if $m=3$} \\ 
    \mathbb{Z}^{d'}\oplus T & \mbox{if $m=4$} \\ 
    \mathbb{Z}^{d} & \mbox{if $m=5$} \\ 
    0 & \mbox{otherwise} 
  \end{array}\right. 
\end{equation}
where $d \geq 1$, $d' \geq 0$ and $T$ is a finite abelian group. For dimensional reasons, all the attaching maps of cells in $\overline{N}$ are in the stable range. Hence, $\overline{N} \simeq \Sigma N'$, for some $CW$-complex $N'$. In particular, $\overline{N}$ is a co-$H$ space. 

Assume that there is a map $h:N \rightarrow S^3$ with a right homotopy inverse $s:S^3 \rightarrow N$. For dimensional reasons, the retraction of $S^3$ off $N$ implies $S^3$ retracts off $\overline{N}$. As $\overline{N}$ is a co-$H$-space, this implies 
that there is a homotopy equivalence $\overline{N}\simeq S^{3}\vee A$ for some space $A$. Poincar\'{e} duality implies that $\overline{N}$ 
has dimension $5$, so the same is true for $A$. It also implies that in 
$\overline{N}\simeq S^{3}\vee A$ there is a class $y\in H^{5}(A)$ whose cup product 
with the cohomology class corresponding to $S^{3}$ equals the generator of $H^{8}(N)$. In \cite[Theorem 1.1]{ST}, it is shown that there is a space $B$ satisfying a homotopy cofibration 
\(\nameddright{B}{}{A}{}{S^{5}}\), 
where $S^{5}$ corresponds to $y$, and the following homotopy decomposition.

\begin{theorem}
\label{thm:decompforM}
    Let $M$ be a simply-connected $6$-dimensional Poincar\'{e} duality complex. Assume that $H^2(M;\mathbb{Z}) \cong \mathbb{Z} \oplus \mathbb{Z}$. Let $N$ be as in Lemma \ref{lem:initialdecomp6dim}. Suppose there is a map $N \to S^3$ which has a right homotopy inverse. Then there is a homotopy equivalence \[\Omega M \simeq S^1 \times S^1 \times \Omega S^3 \times \Omega (A \vee (B \wedge \Omega S^3)).\]
\end{theorem}

\begin{remark}\label{rem:decompforM} 
    By \cite[Example 6.4]{T2}, the hypothesis that there is an $S^{3}$ retracting off $N$ always holds when localised away from $2$ and $3$, in which case Theorem \ref{thm:decompforM} also always holds.
\end{remark}

The decomposition in Theorem \ref{thm:decompforM} can be further refined when localised away from $2$ and $3$,  provided $H_*(M;\mathbb{Z})$ is torsion-free and $H_3(M;\mathbb{Z})\cong 0$. (It is worth noting that when $M$ satisfies the stronger property of being a closed manifold and  $H_3(M;\mathbb{Z})$ is torsion-free, by \cite{W} there is a diffeomorphism $M \cong M' \conn (S^3 \times S^3)^{\conn k}$ where $M'$ is a simply-connected $6$-dimensional manifold with $H_3(M';\mathbb{Z}) = 0$. It can then be shown using \cite[Theorem 9.1]{T1} that the attaching map of the top cell of $M$ is integrally inert and a loop space decomposition can be obtained.)

We begin by identifying the homotopy type of $\overline{N}$ localised away from $2$.
\begin{lemma}
\label{lem:skelN}
If $N$ is defined as in Lemma \ref{lem:initialdecomp6dim}, then localised away from $2$, $\overline{N}$ is homotopy equivalent to a wedge of spheres and Moore spaces.
\end{lemma}
\begin{proof}
    Since $\overline{N}$ is simply-connected, it has a homology decomposition, which is a sequence of homotopy cofibrations \[M_t \xrightarrow{f_t} N_{t-1} \rightarrow N_t,\] for $2 \leq t \leq 7$, with $N_7 = \overline{N}$, each $M_t$ is a wedge of $(t-1)$-dimensional spheres and $t$-dimensional Moore spaces, and $f_t$ is homologically trivial. Since $\overline{N}$ is $2$-connected, $N_1,N_{2} = *$. Moreover, since $N$ is a $2$-connected, $8$-dimensional Poincar\'{e} duality complex, $N_5=N_6=N_7 = \overline{N}$. 
    
     The description of $H_{3}(N;\mathbb{Z})$ in~(\ref{8dimcohomology}) implies that $N_3 \simeq \bigvee_{i=1}^d S^3 \vee P^4(T)$. The description of $H_{4}(N;\mathbb{Z})$ in~(\ref{8dimcohomology}) implies that there is a homotopy cofibration \[\bigvee\limits_{i=1}^{d'} S^3 \vee P^4(T) \xrightarrow{f_4} \bigvee\limits_{i=1}^{d} S^3 \vee P^4(T) \rightarrow N_4.\] Since $f_4$ induces the zero map in homology, the Hurewicz isomorphism implies that $f_{4}$ is null homotopic when restricted to the $3$-skeleton of its domain. Thus $f_{4}$ factors as a composite 
\[\bigvee\limits_{i=1}^{d'} S^{3}\vee P^{4}(T)\xrightarrow{p_2} P^{4}(T)\xrightarrow{q} S^{4}\xrightarrow{g_4}\bigvee\limits_{i=1}^{d} S^{3}\vee P^4(T)\] 
where $p_2$ is the pinch map to the right wedge summand, $q$ collapses out the $3$-skeleton of $P^4(T)$ and $g_4$ is some map. The Hilton-Milnor Theorem implies that $f_4$ is determined by the pinch maps onto each wedge summmand. 
As we are localised away from $2$, it is well known that $\pi_{4}(S^{3})\cong 0$ and, by~\cite[Lemma 3.3]{S} for example, $\pi_{4}(P^{4}(T))\cong 0$. Thus $g_{4}$ is null homotopic, implying that $f_{4}$ is null homotopic. Hence, $N_4 \simeq \bigvee_{i=1}^d S^3 \vee P^4(T) \vee \bigvee_{i=1}^{d'} S^4 \vee P^5(T)$.

    Finally, the description of $H_{5}(N;\mathbb{Z})$ in~(\ref{8dimcohomology}) implies that there is a homotopy cofibration \[\bigvee\limits_{i=1}^d S^4 \xrightarrow{f_{5}} \bigvee_{i=1}^d S^3 \vee P^4(T) \vee \bigvee_{i=1}^{d'} S^4 \vee P^5(T) \rightarrow N_{5}\simeq\overline{N}.\] The Hilton-Milnor theorem implies that $f_{5}$ is determined by the pinch map onto each wedge summand. Since $f_{5}$ induces the zero map in homology, the Hurewicz isomorphism implies that $f_{5}$ composed with the pinch map onto an $S^4$ summand or $P^5(T)$ summand is null homotopic. As in the previous paragraph, since we have localised away from $2$, the pinch map onto an $S^{3}$ or $P^{4}(T)$ summand is also null homotopic. Hence, $f_{5}$ is null homotopic, implying that 
\[\overline{N}\simeq\bigvee_{i=1}^d S^3 \vee P^4(T) \vee \bigvee_{i=1}^{d'} S^4 \vee P^5(T)\vee\bigvee_{i=1}^{d} S^{5}.\] 
\end{proof}


\begin{theorem}
\label{thm:refineddecomp}
     Let $M$ be a simply-connected $6$-dimensional Poincar\'{e} duality complex. Assume that $H^2(M;\mathbb{Z}) \cong \mathbb{Z} \oplus \mathbb{Z}$, $H_3(M;\mathbb{Z})\cong 0$, and $H_*(M;\mathbb{Z})$ is torsion-free. Then localised away from $2$ and~$3$, there is a homotopy equivalence \[\Omega M \simeq S^1 \times S^1 \times \Omega S^3 \times \Omega (A \vee (B \wedge \Omega S^3))\] where $A$ and $B$ are wedges of spheres and Moore spaces.
\end{theorem} 

\begin{proof} 
 Localise away from $2$ and $3$. By Theorem~\ref{thm:decompforM} and Remark~\ref{rem:decompforM}, there is a homotopy equivalence $\Omega M \simeq S^1 \times S^1 \times \Omega S^3 \times \Omega (A \vee (B \wedge \Omega S^3))$. 
By definition of $A$, there is a homotopy equivalence $\overline{N}\simeq S^{3}\vee A$, so as $\overline{N}$ is homotopy equivalent to a wedge of spheres and Moore spaces by Lemma~\ref{lem:skelN}, so is $A$. By definition of $B$, there is a homotopy cofibration $B\rightarrow A\rightarrow S^{5}$ and, as in~\cite[Corollary~3.2]{ST}, the fact that $A$ is homotopy equivalent to a wedge of spheres and Moore spaces implies that this homotopy cofibration splits so that $B$ is also homotopy equivalent to a wedge of spheres and Moore spaces.
\end{proof}

 The analysis can be taken a step further to examine the torsion that appears in $H_{\ast}(\Omega M;\mathbb{Z})$. To begin, we calculate $H_*(N;\mathbb{Z})$. Recall that $H^2(M;\mathbb{Z}) \cong \mathbb{Z} \oplus \mathbb{Z}$. Let $x_1,x_2 \in H^2(M;\mathbb{Z})$ be generators of the first and second $\mathbb{Z}$-summand respectively. The Universal Coefficient Theorem and Poincar\'{e} duality imply $H^2(M;\mathbb{Z}) \cong H^4(M;\mathbb{Z}) \cong \mathbb{Z} \oplus \mathbb{Z}$. Let $y_1,y_2 \in H^4(M;\mathbb{Z})$ be generators of the first and second $\mathbb{Z}$ summand respectively. In cohomology, $x_1^2 = ay_1+by_2$ and $x_2^2 = cy_1+dy_2$, $x_1x_2 = ey_1+fy_2$ for some integers $a,b,c,d,e,f$.

\begin{lemma}
\label{lem:homologyN}
Suppose that $H_*(M;\mathbb{Z})$ is torsion-free, $H_2(M;\mathbb{Z})$ is of rank $2$, and $H_3(M;\mathbb{Z})\cong 0$. Let $U = \begin{pmatrix}a & e & c \\ b & f & d\end{pmatrix}$ and $V= \begin{pmatrix}-e & -f & a & b \\ -c & -d & e & f\end{pmatrix}^T$. Then $H^3(N;\mathbb{Z}) \cong \ker(U)$, $H^4(N;\mathbb{Z}) \cong \mathrm{coker}(U) \oplus \ker(V)$ and $H^5(N;\mathbb{Z}) \cong \mathrm{Tor}(\mathrm{coker}(U)) \oplus \ker(U)$, where $\mathrm{Tor}$ denotes the torsion subgroup. Moreover, $H_*(N)$ contains $p$-torsion if and only if $\mathrm{coker}(U)$ contains $p$-torsion.
\end{lemma}
\begin{proof}
    Cohomology will be taken with $\mathbb{Z}$ coefficients throughout. We calculate $H^*(N)$ as a module using the Serre spectral sequence applied to the homotopy fibration $S^1 \times S^1 \to N \to M$. The $E_2$ page has the form $E_2^* \cong H^*(M) \otimes H^*(S^1 \times S^1)$. Let 
\begin{align*} 
    H^2(M) & \cong \mathbb{Z}\{x_1,x_2\} & H^1(S^1 \times S^1) & \cong \mathbb{Z}\{\epsilon_1,\epsilon_2\} \\
    H^4(M) & \cong \mathbb{Z}\{y_1,y_2\} & H^2(S^1 \times S^1) & \cong \mathbb{Z}\{\epsilon_1\epsilon_2\}. \\ 
    H^6(M) & \cong \mathbb{Z}\{z\} & & 
  \end{align*} 

Poincar\'e Duality implies the map $\phi: H^2(M) \to \mathrm{Hom}(H^4(M),H^6(M))$ sending $x \in H^2(M)$ to the homomorphism defined by $\phi(x_i)(y) = xy$ for $y \in H^4(M)$ is an isomorphism. Let $y_1$ and $y_2$ be a basis of $H^4(M)$, and let $y_1^*$, and $y_2^*$ be the dual basis of $\mathrm{Hom}(H^4(M),H^6(M))$. Since $\phi$ is an isomorphism, there exist basis elements $x_1, x_2 \in H^2(M)$ such that $\phi(x_1)=y_1^*$ and $\phi(x_2)=y_2^*$. By definition of $\phi$, we obtain $x_iy_j = \phi(x_i)(y_j) = y_i^*(y_j) = \delta_{ij}z$. Therefore, we can assume $x_1 y_1 = x_2y_2=z$, and $x_1y_2=x_2y_1=0$. 

    For dimensional reasons, the only possibly non-trivial differentials are $d_2$, hence $E_3 \cong E_\infty$. By Lemma \ref{lem:initialdecomp6dim}, $N$ is $2$-connected. It follows that without loss of generality that the generators $\epsilon_1,\epsilon_2$ can be chosen so that $d_2(\epsilon_1)=x_1$ and $d_2(\epsilon_2) = x_2$. By~(\ref{8dimcohomology}), the only possibly non-trivial cohomology groups of $N$ are $H^3(N) \cong \mathbb{Z}^{d}$, $H^4(N) \cong \mathbb{Z}^{d'} \oplus T$, $H^5(N) \cong \mathbb{Z}^{d} \oplus T$, and $H^8(N) \cong \mathbb{Z}$, where $T$ is a finite abelian group. Hence, it suffices to calculate $H^3(N)$ and $H^4(N)$.
    
     First, we calculate $H^3(N)$. Consider $d_2^{0,2}: E_2^{0,2} = \mathbb{Z}\{\epsilon_1\epsilon_2\} \to E_2^{2,1} = \mathbb{Z}\{\epsilon_1x_1,\epsilon_1x_2,\epsilon_2x_1,\epsilon_2x_2\}$. Since $d_2$ is a derivation, \[d_2(\epsilon_1\epsilon_2) = \epsilon_2x_1 - \epsilon_1x_2.\] Next consider $d_2^{2,1}: E_2^{2,1} = \mathbb{Z}\{\epsilon_1x_1,\epsilon_1x_2,\epsilon_2x_1,\epsilon_2x_2\} \to E_2^{4,0} = \mathbb{Z}\{y_1,y_2\}$. For dimensional reasons, $d_2(x_1) = d_2(x_2) = 0$. Recall $x_1^2 = ay_1+by_2$, $x_2^2 = cy_1+dy_2$, and $x_1x_2 = ey_1+fy_2$ for some integers $a,b,c,d,e,f$. As $d_{2}$ is a derivation, we obtain 
\begin{align*} 
  d_2(\epsilon_1x_1) & =x_1^2= ay_1+by_2 & d_2(\epsilon_1x_2) & = x_1x_2 =ey_1+fy_2 \\ 
  d_2(\epsilon_2x_1) & =x_2x_1 = ey_1+fy_2 &  d_2(\epsilon_2x_2) & =x_2^2= cy_1+dy_2. 
\end{align*} 
     Let $U=\begin{pmatrix}a & e & c \\ b & f & d\end{pmatrix}$. Putting the calculations together, $E_3^{2,1} \cong \ker(d_2^{2,1})/\mathrm{im}(d_2^{0,2}) \cong \ker(U)$, implying that $H^3(N) \cong \ker(U)$.

     Next, we calculate $H^4(N)$. From the calculation of $d_2^{2,1}$, we obtain $E_3^{4,0} \cong \mathrm{coker}(U)$. Now consider $d_2^{2,2}:E_2^{2,2} = \mathbb{Z}\{\epsilon_1\epsilon_2x_1,\epsilon_1\epsilon_2x_2\} \to E_2^{4,1}= \mathbb{Z}\{\epsilon_1y_1,\epsilon_2y_1,\epsilon_1y_2,\epsilon_2y_2\}$. For dimensional reasons, $d(y_1)=d(y_2) = 0$. Since $d_2$ is a derivation, $x_1y_1 = x_2y_2=z$, and $x_1y_2 = x_2y_1 =  0$, we obtain \[d_2(\epsilon_1\epsilon_2x_1) = -e\epsilon_1y_1-f\epsilon_1y_2+a\epsilon_2y_1+b\epsilon_2y_2\]\[d_2(\epsilon_1\epsilon_2x_2) = -c\epsilon_1y_1-d\epsilon_1y_2+e\epsilon_2y_1+f\epsilon_2y_2.\] Therefore, if $V = \begin{pmatrix}-e & -f & a & b \\ -c & -d & e & f\end{pmatrix}^T$ then $E^{2,2}_3 \cong \ker(V)$. Since $E^{2,2}_\infty$ is free, we obtain $H^4(N) \cong \mathrm{coker}(U) \oplus \ker(V)$. 

      Finally, the assertion that $H^5(N;\mathbb{Z}) \cong \mathrm{Tor}(\mathrm{coker}(U)) \oplus \ker(U)$ follows from the Universal Coefficient Theorem.
\end{proof} 

    In \cite{BB,BT1}, decompositions of the loops on an $(n-1)$-connected, $(2n)$-dimensional Poincar\'{e} duality complex under mild cohomological conditions is given. Similar decompositions are given in \cite{B,BT2,BW,HT} for certain families of $(n-1)$-connected, $(2n+1)$-dimensional Poincar\'{e} duality complexes. A consequence of these decompositions is that the torsion in $H_*(\Omega M;\mathbb{Z})$ is the same as the torsion in $H_*(M;\mathbb{Z})$. In the case of a simply-connected $6$-dimensional Poincar\'{e} duality complex~$M$ the torsion in $H_{\ast}(\Omega M;\mathbb{Z})$ can also depend on the cup product structure. 
    
\begin{proposition} 
  \label{6dimlooptorsion} 
  Let $M$ be a simply-connected $6$-dimensional Poincar\'{e} duality complex. Assume that $H^2(M;\mathbb{Z}) \cong \mathbb{Z} \oplus \mathbb{Z}$, $H_3(M;\mathbb{Z})\cong 0$, and $H_*(M;\mathbb{Z})$ is torsion-free. With $N$ as in Lemma \ref{lem:initialdecomp6dim}, suppose there is a map $N \to S^3$ which has a right homotopy inverse. Then the torsion in $H_*(\Omega M;\mathbb{Z})$ is determined by the cohomology ring structure of $H^*(M;\mathbb{Z})$. 
\end{proposition} 

\begin{proof} 
 By Theorem~\ref{thm:decompforM}, $\Omega M\simeq S^{1}\times S^{1}\times\Omega S^{3}\times\Omega(A\vee(B\wedge\Omega S^{3})$. Thus the torsion in $H_{\ast}(\Omega M;\mathbb{Z})$ is determined by the torsion in $H_{\ast}(A;\mathbb{Z})$ and $H_{\ast}(B;\mathbb{Z})$. By definition, $A$ is a retract of $\overline{N}$ and $B$ satisfies a homotopy cofibration $B\rightarrow A\rightarrow S^{5}$. By its construction in~\cite{ST}, this cofibration splits in homology. Thus the torsion in $H_{\ast}(A;\mathbb{Z})$ and $H_{\ast}(B;\mathbb{Z})$ is determined by the torsion in $H_{\ast}(N;\mathbb{Z})$. As  $H^2(M;\mathbb{Z}) \cong \mathbb{Z} \oplus \mathbb{Z}$, $H_3(M;\mathbb{Z})\cong 0$, and $H_*(M;\mathbb{Z})$ is torsion-free, Lemma \ref{lem:homologyN} implies that the torsion in $H_*(N;\mathbb{Z})$ is determined by the cohomology ring structure of $H^*(M;\mathbb{Z})$.
\end{proof}

\section{An application to Moore's Conjecture}
\label{sec:mooreskeleton}

In this section, we use Theorem~\ref{thm:localinert} to show that the $(n-1)$-skeleton 
$\overline{M}$ of $M$ satisfies the hyperbolic form of Moore's Conjecture
after localising away from a finite set of primes. We begin by recalling the
relevant terminology. A simply-connected finite \(CW\)-complex \(X\) is
\emph{rationally elliptic} if \(\pi_\ast(X)\otimes\mathbb{Q}\) is finite
dimensional, and is \emph{rationally hyperbolic} otherwise. If \(p\) is a
prime, the \emph{\(p\)-primary homotopy exponent} of \(X\), denoted
\(\exp_p(X)\), is the least power of \(p\)
which annihilates all \(p\)-torsion in \(\pi_\ast(X)\), if such a power exists;
otherwise \(\exp_p(X)=\infty\). \emph{Moore's Conjecture} asserts that the 
following are equivalent: (i) $X$ is rationally elliptic, (ii) \(\exp_p(X)\) is finite for some prime $p$,  
and (iii) $\exp_p(X)$ is finite for all primes $p$. The hyperbolic form of the conjecture states that 
the following are equivalent: (i) $X$ is rationally hyperbolic, \(\exp_p(X)=\infty\) for some prime $p$, 
and (iii) $\exp_p(X)=\infty$ for all primes $p$. 

We also recall the local hyperbolicity notion introduced by Huang and
Wu~\cite{HW}. For a prime \(p\) and \(r\geq 1\), let \(t_N^{(r)}(X)\) be the
number of \(\mathbb{Z}/p^r\mathbb{Z}\)-summands in
\(\oplus_{i\leq N}\pi_i(X)\). The space \(X\) is
\emph{\(\mathbb{Z}/p^r\mathbb{Z}\)-hyperbolic} if
\[
\liminf_{N\to\infty}\frac{\log t_N^{(r)}(X)}{N}>0.
\]

When this condition holds for every \(r\geq 1\), it implies
\(\exp_p(X)=\infty\). Thus such local hyperbolicity statements give
quantitative strengthenings of the hyperbolic direction of Moore's Conjecture.
The conjecture is known in many special cases but remains open in general, and
recent progress often proves such large-prime or local hyperbolicity forms;
see, for example,~\cite{HW,HT2,H2}.

For a homotopy cofibration \(\Sigma A\xrightarrow{u}Y\xrightarrow{v}Z\), we
say that \(u\) is \emph{inert} if \(\Omega v\) has a right homotopy inverse.
The following input from~\cite{HT2} is stated in the form needed here. If
\(X\) is a path-connected finite \(CW\)-complex of dimension \(d\) and
connectivity \(s\), let \(\mathcal{P}_{\mathrm{free}}(X)\) be the set of
primes \(q\) such that
\[
q\leq \frac{d-s+1}{2}
\]
or \(H_\ast(X;\mathbb{Z})\) has \(q\)-torsion.

\begin{theorem}[{\cite[Theorem~5.2 and Corollary~5.3]{HT2}}]
\label{thm:freeloopMoore}
Let
\[
\Sigma A\xrightarrow{u}Y\xrightarrow{v}Z
\]
be a homotopy cofibration in which \(A\) is path-connected,
\(\Sigma A\), \(Y\), and \(Z\) are simply-connected finite \(CW\)-complexes,
and \(A\) and \(Z\) are not rationally contractible. If \(u\) is inert, then
after localising away from
\[
\mathcal{P}_{\mathrm{free}}(A)\cup\mathcal{P}_{\mathrm{free}}(Z)
\]
there is a homotopy retraction of \(\Omega(S^a\vee S^b)\) off \(\Omega Y\),
for some \(a,b\geq 2\). Consequently, \(Y\) is rationally hyperbolic and, for
every prime \(p\) not in this set, \(Y\) has no homotopy exponent at \(p\) and
is \(\mathbb{Z}/p^r\mathbb{Z}\)-hyperbolic for all \(r\geq 1\). \qed
\end{theorem}

Let \(M\) satisfy the hypotheses of Theorem~\ref{thm:localinert}. Let
\(\mathcal{T}(M)\) be the set of primes \(q\) for which
\(H_\ast(M;\mathbb{Z})\) has \(q\)-torsion. Define a number
\(B_{\mathrm{in}}(M)\) as follows. If \(m\) is odd, let
\[
B_{\mathrm{in}}(M)=\frac{n-m+3}{2}.
\]
If \(m\) is even, let
\[
\mathcal{S}(M)=\left\{j \,\middle|\,
\begin{array}{l}
j\text{ is even},\, H_j(M;\mathbb{Z})\neq 0,\,
H_{k}(M;\mathbb{Z})=0\\
\text{for all odd } k\leq j, \, m\leq j\leq 2m-2
\end{array}\right\},
\]
let \(s=\max(\mathcal{S}(M))\), let
\(r_k=\mathrm{rank}(H_k(M;\mathbb{Z}))\), and let
\[
B_{\mathrm{in}}(M)=
\frac{n+\sum_{k\in\mathcal{S}(M),\,k\neq s} r_k(k-1)+(r_s-1)(s-1)+3}{2}
-\frac{s}{2}.
\]
In either case, let
\[
\mathcal{I}(M)=\left\{p \,\middle|\, p\leq B_{\mathrm{in}}(M)\right\}
\]
be the set of primes excluded by Theorem~\ref{thm:localinert}, and define
\[
\mathcal{Q}(M)=\mathcal{I}(M)\cup\mathcal{T}(M).
\]

\begin{theorem}
\label{thm:skeletonretract}
Let \(M\) be an \((m-1)\)-connected, \(n\)-dimensional Poincar\'{e} duality
complex satisfying the hypotheses of Theorem~\ref{thm:localinert}, and let
\[
S^{n-1}\xrightarrow{f}\overline{M}\xrightarrow{i}M
\]
be the homotopy cofibration determined by attaching the top cell of \(M\). Then, after
localising away from \(\mathcal{Q}(M)\), there is a retraction of
\[
\Omega(S^{n-1}\vee S^{n+m-2})
\]
off \(\Omega\overline{M}\).
\end{theorem}

\begin{proof}
By Theorem~\ref{thm:localinert}, after localising away from
\(\mathcal{I}(M)\), the attaching map \(f\) is inert. Equivalently,
\(\Omega i\) has a right homotopy inverse. Write the cofibration attaching 
the top cell to $M$ as
\[
\Sigma S^{n-2}\xrightarrow{f}\overline{M}\xrightarrow{i}M.
\]
We apply Theorem~\ref{thm:freeloopMoore} with
\(A=S^{n-2}\), \(Y=\overline{M}\), and \(Z=M\). Since \(n>m\geq 2\),
\(\Sigma S^{n-2}=S^{n-1}\), \(\overline{M}\), and \(M\) are simply-connected,
while \(S^{n-2}\) and \(M\) are not rationally contractible. Moreover
\(\mathcal{P}_{\mathrm{free}}(S^{n-2})=\emptyset\), since
\((n-2)-(n-3)+1=2\) and \(S^{n-2}\) has torsion-free homology.

It remains to compare the finite sets of primes. Since \(M\) is
\((m-1)\)-connected and \(n\)-dimensional,
\[
\mathcal{P}_{\mathrm{free}}(M)=
\left\{p \,\middle|\, p\leq \frac{n-m+2}{2}\right\}\cup\mathcal{T}(M).
\]
The numerical bound \(B_{\mathrm{in}}(M)\) is at least \((n-m+2)/2\). If \(m\)
is odd this is immediate. If \(m\) is even, the difference between
\(B_{\mathrm{in}}(M)\) and \((n-m+2)/2\) is
\[
\frac{
\sum_{k\in\mathcal{S}(M),\,k\neq s}r_k(k-1)+(r_s-1)(s-1)+m+1-s
}{2}.
\]
This number is positive: if \(s=m\) it equals
\(\big((r_m-1)(m-1)+1\big)/2\), while if \(s>m\) the sum contains the term
\(r_m(m-1)\) and \(s\leq 2m-2\). Therefore localising away from
\(\mathcal{Q}(M)\) is enough both for inertness and for the suspension
splitting required in Theorem~\ref{thm:freeloopMoore}.

The proof of~\cite[Theorem~5.2]{HT2} now identifies the two spheres that
occur. The first is the sphere retracting off \(\Sigma S^{n-2}\), namely
\(S^{n-1}\). For the second, one considers
\[
\Omega M\wedge S^{n-1}\simeq \Sigma^{n-1}\Omega M
\xrightarrow{\Sigma^{n-2}\mathrm{ev}}\Sigma^{n-2}M 
\] where $ev$ is the canonical evaluation map. 
After localising away from \(\mathcal{Q}(M)\), the suspension splitting
of~\cite[Lemma~5.1]{HT2} implies that \(\Sigma M\), and hence
\(\Sigma^{n-2}M\), is homotopy equivalent to a wedge of spheres. Since \(M\) is
\((m-1)\)-connected and the hypotheses of Theorem~\ref{thm:localinert} imply
that \(H_m(M;\mathbb{Z})\) has a nonzero free summand, this wedge
decomposition contains an \(S^{n+m-2}\)-summand.
The argument in~\cite[Theorem~5.2]{HT2} then shows that
\(S^{n+m-2}\) retracts off \(\Omega M\wedge S^{n-1}\).

Finally, since \(\Omega i\) has a right homotopy inverse, the decomposition
used in~\cite[Theorem~5.2]{HT2} gives
\[
\Omega\overline{M}\simeq
\Omega M\times\Omega\bigl((\Omega M\wedge S^{n-1})\vee S^{n-1}\bigr).
\] 
As \(S^{n+m-2}\) and \(S^{n-1}\) retracts off $(\Omega M\vee S^{n-1})\vee S^{n-1}$, 
we obtain a retraction of
\(\Omega(S^{n-1}\vee S^{n+m-2})\) off \(\Omega\overline{M}\).
\end{proof}

\begin{corollary}
\label{cor:mooreskeleton}
With the hypotheses of Theorem~\ref{thm:skeletonretract},
\(\overline{M}\) is rationally hyperbolic. Moreover, for every prime
\(p\notin\mathcal{Q}(M)\), \(\exp_p(\overline{M})=\infty\) and
\(\overline{M}\) is \(\mathbb{Z}/p^r\mathbb{Z}\)-hyperbolic for all
\(r\geq 1\). Consequently, \(\overline{M}\) satisfies the hyperbolic form of
Moore's Conjecture away from the finite set of primes \(\mathcal{Q}(M)\).
\end{corollary}

\begin{proof}
This is immediate from Theorem~\ref{thm:skeletonretract} and
\cite[Corollary~5.3]{HT2}.
\end{proof}

Corollary~\ref{cor:mooreskeleton} also gives further evidence for
\cite[Conjecture~1.6]{HT2} and \cite[Conjecture~1.7]{H2}. These conjectures
are partial strengthenings of the hyperbolic direction of Moore's Conjecture:
they ask not only for \(\exp_p(X)=\infty\) at the relevant primes, but for
exponential growth of \(p\)-torsion in homotopy groups.

\begin{remark}
The method introduces no additional numerical bound beyond
Theorem~\ref{thm:localinert}; the only further exclusions are the torsion
primes in \(H_\ast(M;\mathbb{Z})\). Thus if \(H_\ast(M;\mathbb{Z})\) is
torsion-free, then \(\mathcal{Q}(M)=\mathcal{I}(M)\).
\end{remark}

\section{Cohen--Neisendorfer spaces and non-inert attaching maps}
\label{sec:CNnoninert}

We end with an application of Corollary~\ref{cor:mooreskeleton} in the
opposite direction. The point is that the local hyperbolicity of the lower
skeleton of a Cohen--Neisendorfer finite \(H\)-space forces certain maps
between spheres to be non-inert. This is useful because, while inertness has
received considerable attention, comparatively little is known about
non-inertness. The examples below extend the Moore-space-based phenomena
considered in~\cite{H3}.

We first recall the growth estimate for spheres that will be used to detect
non-inertness. The following formulation follows from the subexponential
growth theorem of Burklund--Senger~\cite{BS}.

\begin{theorem}[{\cite{BS}}]
\label{thm:BSspheres}
Let \(p\) be a prime and \(d\geq 2\). For every \(r\geq 1\), the number of
\(\mathbb{Z}/p^r\mathbb{Z}\)-summands in
\[
\bigoplus_{i\leq N}\pi_i(S^d)
\]
grows subexponentially as \(N\to\infty\). \qed
\end{theorem}

Fix an odd prime \(p\), and work \(p\)-locally. Let
\[
S^{b-1}\xrightarrow{\alpha}S^a\longrightarrow A
\]
be a homotopy cofibration with \(3\leq a<b\) where $a,b$ are odd. Then \(A\) is a
simply-connected two-cell complex whose cells are in odd dimensions. If
\(p>(b+3)/2\), then in particular \(p\geq 5\), so the two-cell complex~\(A\)
has fewer than \(p-1\) cells. The Cohen--Neisendorfer construction~\cite{CN}
therefore gives a finite \(p\)-local \(H\)-space \(B(A)\) and a map
\[
j\colon A\longrightarrow B(A)
\]
such that \(j_\ast(\widetilde H_\ast(A;\mathbb{Z}_{(p)}))\) is the submodule of
algebra generators in
\[
H_\ast(B(A);\mathbb{Z}_{(p)})\cong
\Lambda_{\mathbb{Z}_{(p)}}(x_a,x_b).
\]
In this rank two case, \(B(A)\) is a three-cell complex whose lower skeleton is
\(A\); equivalently, there is a homotopy cofibration
\[
S^{a+b-1}\xrightarrow{f}A\xrightarrow{j}B(A).
\]
See also \cite[Theorem~2.1]{GHMTW} for this rank two description.

\begin{theorem}
\label{thm:CNnoninert}
Let \(p\) be an odd prime, and let
\[
\alpha\colon S^{b-1}\longrightarrow S^a
\]
be a \(p\)-local map with \(3\leq a<b\) where $a,b$ are odd. If \(p>(b+3)/2\), then
\(\alpha\) is not inert at $p$. 
\end{theorem}

\begin{proof}
Let \(A\) be the homotopy cofibre of \(\alpha\), and let \(B(A)\) be the
Cohen--Neisendorfer finite \(H\)-space associated to \(A\). The exterior
algebra description makes \(B(A)\) a \(p\)-local Poincar\'{e} duality space of
formal dimension \(a+b\). Since \(A\) is the lower skeleton, \(B(A)\) is
\((a-1)\)-connected and \(H_a(B(A);\mathbb{Z}_{(p)})\) has a free
\(\mathbb{Z}_{(p)}\)-summand. The inertness bound in
Theorem~\ref{thm:localinert}, applied \(p\)-locally to \(B(A)\), is
\[
\frac{(a+b)-a+3}{2}=\frac{b+3}{2}.
\]
Thus the top cell attaching map \(f\) of \(B(A)\) is inert at \(p\). Since
\(H_\ast(B(A);\mathbb{Z}_{(p)})\) is torsion-free, the \(p\)-local form of
Corollary~\ref{cor:mooreskeleton} implies that \(A\) is
\(\mathbb{Z}/p^r\mathbb{Z}\)-hyperbolic for every \(r\geq 1\).

Suppose, for a contradiction, that \(\alpha\) is inert at \(p\). Then the loop
map
\[
\Omega S^a\longrightarrow \Omega A
\]
has a right homotopy inverse. Hence \(\Omega A\) is a retract of
\(\Omega S^a\). It follows that, for every \(r\geq 1\), the number of
\(\mathbb{Z}/p^r\mathbb{Z}\)-summands in \(\oplus_{i\leq N}\pi_i(A)\) is
bounded above by the corresponding number for \(S^a\). This is subexponential
by Theorem~\ref{thm:BSspheres}, contradicting the
\(\mathbb{Z}/p^r\mathbb{Z}\)-hyperbolicity of \(A\). Therefore \(\alpha\) is
not inert at \(p\).
\end{proof}

\begin{remark}
Theorem~\ref{thm:CNnoninert} produces non-inert maps between spheres from
finite \(H\)-spaces, rather than from Moore spaces. In this sense it is
complementary to the examples in~\cite{H3}, where the non-inertness phenomena
come from Moore space summands. Here the obstruction is asymptotic: inertness
would force the \(p\)-torsion growth of \(A\) to be bounded by that of a sphere,
but Corollary~\ref{cor:mooreskeleton} makes it exponential.
\end{remark}

\bibliographystyle{amsalpha}

\end{document}